\documentclass[11pt, a4paper]{amsart}

\usepackage{a4}
\usepackage{a4wide}
\usepackage{amssymb}
\usepackage{amsmath}
\usepackage{amsthm}
\usepackage{amstext}
\usepackage{amscd}
\usepackage{latexsym}
\usepackage{graphics}
\usepackage{color}
\usepackage{enumitem}
\usepackage[all]{xy}
\usepackage[textsize=scriptsize]{todonotes}

\usepackage[colorlinks, pagebackref]{hyperref}
\hypersetup{
  colorlinks=true,
  citecolor=blue,
  linkcolor=blue,
  urlcolor=blue}

\makeatletter
\def\@tocline#1#2#3#4#5#6#7{\relax
  \ifnum #1>\c@tocdepth 
  \else
    \par \addpenalty\@secpenalty\addvspace{#2}%
    \begingroup \hyphenpenalty\@M
    \@ifempty{#4}{%
      \@tempdima\csname r@tocindent\number#1\endcsname\relax
    }{%
      \@tempdima#4\relax
    }%
    \parindent\z@ \leftskip#3\relax \advance\leftskip\@tempdima\relax
    \rightskip\@pnumwidth plus4em \parfillskip-\@pnumwidth
    #5\leavevmode\hskip-\@tempdima
      \ifcase #1
      \or\or \hskip 2em \or \hskip 2homologyem \else \hskip 3em \fi%
      #6\nobreak\relax
    \dotfill\hbox to\@pnumwidth{\@tocpagenum{#7}}\par
    \nobreak
    \endgroup
  \fi}
\makeatother

\theoremstyle{plain}

\newtheorem{theorem}{Theorem}[section]

\newtheorem{lemma}[theorem]{Lemma}
\newtheorem{corollary}[theorem]{Corollary}
\newtheorem{proposition}[theorem]{Proposition}
\newtheorem{conjecture}[theorem]{Conjecture}
\theoremstyle{definition}

\newtheorem{notation}[theorem]{Notation}
\newtheorem{remark}[theorem]{Remark}

\newtheorem{definition}[theorem]{Definition}
\newtheorem{example}[theorem]{Example}

\numberwithin{equation}{section}

\newcommand{\Hom}{{\rm Hom}}

\newcommand{\Spec}{{\rm Spec \,}}

\newcommand{\Gal}{{\rm Gal}}

\newcommand{\N}{{\mathbb N}}

\newcommand{\Z}{{\mathbb Z}}
\newcommand{\A}{{\mathbb A}}

\newcommand{\KMW}{{\mathbf K}^{\rm MW}}		

\newcommand{\Ccell}{{C}^{\rm cell}}			
\newcommand{\Hcell}{{\mathbf H}^{\rm cell}} 
\newcommand{\HAred}{\~{\mathbf H}^{\A^1}}  	
\newcommand{\HA}{{\mathbf H}^{\A^1}}  	
\newcommand{\ZA}{{\mathbf Z}_{\A^1}}   

\def\<{\langle}
\def\>{\rangle} 
\def\-{\overline} 
\def\~{\widetilde}
\def\^{\widehat}

\def\@{\mathcal}
\def\!{\mathscr}
\def\#{\mathbb}
\def\_{\underline}

\input{xy}
\xyoption{all}

\begin{document}

\title[Transfers on $\A^1$-connected components and the norm principle]{Transfers on $\A^1$-connected components of quasi-split groups and the norm principle}

\author{Amit Hogadi}
\address{Department of Mathematical Sciences, Indian Institute of Science Education and Research Pune, Dr. Homi Bhabha Road, Pashan, Pune 411008, India.}
\email{amit@iiserpune.ac.in}

\author{Anand Sawant}
\address{School of Mathematics, Tata Institute of Fundamental Research, Homi Bhabha Road, Colaba, Mumbai 400005, India.}
\email{asawant@math.tifr.res.in}
\date{}

\thanks{The authors acknowledge the support of SERB MATRICS grant MTR/2023/000228, India DST-DFG Project on Motivic Algebraic Topology DST/IBCD/GERMANY/DFG/2021/1 and the Department of Atomic Energy, Government of India, under project no. 12-R\&D-TFR-5.01-0500.}

\begin{abstract} 
We show that the sheaf of $\A^1$-connected components of a quasi-split group over a perfect field is a strictly $\A^1$-invariant sheaf with (Voevodsky) transfers. As a consequence, we show that the norm principle holds for any quasi-split group over a perfect field.
\end{abstract}

\maketitle
\tableofcontents

\setlength{\parskip}{2pt plus1pt minus1pt}

\section{Introduction}
\label{section introduction}

One of the motivations behind this article is to initiate a study of questions around norm principles for algebraic groups via their $\A^1$-connected components. The sheaf of $\A^1$-connected components of an algebraic group contains interesting information pertaining to near-rationality properties of the group because of its relationship with $R$-equivalence classes \cite{Balwe-Hogadi-Sawant-JTop}, \cite{Balwe-Sawant-R-eqivalence-IMRN}.  Versions of norm principles have been introduced and studied in \cite{Gille-1993}, \cite{Gille-1997} and \cite{Merkurjev-norm-principle}, especially in relation with the $R$-triviality of the algebraic group in question.    We will focus on the norm principle formulated by Merkurjev \cite{Merkurjev-norm-principle}, which is a generalization of the classical norm principles in the theory of quadratic forms by Knebusch and by Scharlau (see \cite[Section 3]{barquero-merkurjev}).  The norm principle is known to hold for reductive groups whose Dynkin diagrams do not contain connected components of type $D_n$ for $n \geq 4$, $E_6$ or $E_7$ \cite{barquero-merkurjev}. For groups of type $D_n$, the norm principle has been shown to hold over complete discretely valued fields in \cite{Bhaskhar-Chernousov-Merkurjev}. The norm principle is expected to hold in general.

In view of this connection, a natural question is whether the sheaf of $\A^1$-connected components of a reductive algebraic groups admits appropriate transfers and whether the norm principle can be deduced from this. The aim of this paper is to answer this question in the affirmative for quasi-split groups over perfect fields. Standard arguments (see \cite{Balwe-Sawant-reductive}, \cite{Balwe-Hogadi-Sawant-JTop} for example) reduce the study of $\A^1$-connected components of a reductive algebraic group to that of the simply connected cover of its derived subgroup. The sheaf of $\A^1$-connected components of a split, semisimple, simply connected group is trivial (see \cite[Theorem 5.2]{Balwe-Sawant-reductive}). As a result, quasi-split groups give the first class of algebraic groups whose $\A^1$-connected components are a priori nontrivial. The main result of this paper is as follows (see Theorem \ref{theorem transfers} for a more precise statement).

\begin{theorem}
\label{theorem intro main}
Let $G$ be a quasi-split group over a perfect field $k$.  Then the sheaf $\pi_0^{\A^1}(G)$ is a strictly $\A^1$-invariant sheaf with Voevodsky transfers.
\end{theorem}

We observe that if $\pi_0^{\A^1}(G)$ has transfers for field extensions, then the norm principle holds for $G$ (Theorem \ref{theorem pi0norm}). We thus have the following consequence (see Corollary \ref{corollary norm principle}) for all quasi-split groups over perfect fields.

\begin{corollary}
\label{corollary intro main} 
For any quasi-split group $G$ over a perfect field $k$, the norm principle (see Definition \ref{defn norm principle}) holds with respect to any morphism $G \to H$, where $H$ is a torus.
\end{corollary}

We emphasize that our proofs of our main results are independent of the classification of algebraic groups, unlike previous approaches to the norm principle. The proof makes key use of techniques from $\A^1$-homotopy theory. The main novelty is the use of the theory of cellular $\A^1$-homology developed in \cite{Morel-Sawant-pi1}. The outline of our proof of Theorem \ref{theorem intro main} and the key ingredients are as follows:
\begin{itemize}
\item Fix a Borel subgroup $B$ and the corresponding maximal torus $T$ of $G$. The $\A^1$-fiber sequence $B \to G \to G/B$ induces a long exact sequence
\begin{equation}
\label{eqn intro exact sequence}
\cdots \pi_1^{\A^1}(G/B) \to T \to \pi_0^{\A^1}(G) \to \pi_0^{\A^1}(G/B). 
\end{equation}

\item Theorem \ref{theorem quasi-split} shows that the scheme $G/B$ is $\A^1$-connected using the Bruhat decomposition and the geometric criterion for $\A^1$-connectedness proved in \cite{Balwe-Hogadi-Sawant-JAG}. 

\item Since $G/B$ is $\A^1$-connected, by the $\A^1$-Hurewicz theroem one has $$\pi_1^{\A^1}(G/B) \cong \HA_1(G/B) \cong \Hcell_1(G/B).$$ The terms of the cellular $\A^1$-chain complex of $G/B$ admits field transfers (see Definition \ref{defn field transfers}). 

\item Since any torus is a strictly $\A^1$-invariant sheaf with transfers in the sense of Voevodsky, it remains to show that the morphism $\pi_1^{\A^1}(G/B) \to T$ in \eqref{eqn intro exact sequence} preserves transfers. This is done in Theorem \ref{theorem transfers} by using the identification as a cellular $\A^1$-homology group, reducing to the case of field transfers and observing that any homomorphism from a Weil restriction of $\KMW_1$ to any torus preserves transfers (Lemma \ref{lemma transfers}, Lemma \ref{lemma WR transfers}).
\end{itemize}

It is an open question whether the sheaf of $\A^1$-connected components of any reductive algebraic group admits transfers (or at least, \emph{field transfers} in the sense of Definition \ref{defn field transfers}). We are optimistic that classification-independent proofs of the norm principle in new cases can be obtained via this approach using $\A^1$-homotopy theory.

\subsection*{Acknowledgements}

Anand Sawant thanks Fabien Morel for encouraging him to explore the existence of transfers on $\A^1$-connected components and for helpful discussions. We also thank Alexey Ananyevskiy and Sandeep Varma for their helpful comments and discussions. Hospitality of IISER Pune while this work was being done is gratefully acknowledged.

\section{\texorpdfstring{$\A^1$}{A1}-connectedness of the flag variety of a quasi-split group}
\label{section strict}

Let $k$ be a field. Recall from \cite[Definition 1.7]{Morel-book} that a Nisnevich sheaf of groups $\mathcal F$ on the category $Sm_k$ of essentially smooth, finite type, separated schemes over $k$ is said to be 
\begin{itemize}
\item \emph{strongly $\A^1$-invariant} if the presheaves $H^n_{\rm Nis}(-, \mathcal F)$ are $\A^1$-invariant on $Sm_k$ for $n =0, 1$; 

\item \emph{strictly $\A^1$-invariant} if $\mathcal F$ is a sheaf of abelian groups and the presheaves $H^n_{\rm Nis}(-, \mathcal F)$ are $\A^1$-invariant on $Sm_k$ for every $n \in \N$.
\end{itemize}

A fundamental theorem of Morel \cite[Theorem 1.16]{Morel-book} in $\A^1$-algebraic topology states that strongly $\A^1$-invariant sheaves on $Sm_k$ of abelian groups are strictly $\A^1$-invariant. We will denote the category of strictly $\A^1$-invariant sheaves on $Sm_k$ by $Ab_{\A^1}(k)$. The category of strictly $\A^1$-invariant sheaves on $Sm_k$ is an abelian category.  This nontrivial fact is a consequence of the fact that the category of strictly $\A^1$-invariant sheaves on $Sm_k$ can be identified as the heart of the homological $t$-structure on the $\A^1$-derived category \cite[Lemma 6.2.11]{Morel-connectivity}.   

For a reductive algebraic group $G$ over $k$, it is an open question whether the sheaf $\pi_0^{\A^1}(G)$ is a sheaf of abelian groups.  In view of \cite[Corollary 4.17]{Choudhury}, showing that $\pi_0^{\A^1}(G)(F)$ is an abelian group for any finitely generated, separable field extension $F$ of $k$ is sufficient to show that $\pi_0^{\A^1}(G)$ is a sheaf of abelian groups.  

\begin{conjecture}
\label{conj strict}
For every reductive group $G$ over a field $k$, the sheaf $\pi_0^{\A^1}(G)$ is a sheaf of abelian groups. Moreover, $\pi_0^{\A^1}(G)$ is strictly $\A^1$-invariant. 
\end{conjecture}

Recall that a \emph{torus} over a field $k$ is an algebraic group over $k$, which on base change to an algebraic closure of $k$ is isomorphic to a product of $\mathbb G_m$'s.  Tori are strongly (and consequently, strictly) $\A^1$-invariant.  We record a proof of this simple fact below for the sake of completeness.

\begin{lemma}
\label{lemma torus}
Let $T$ be torus over $k$. Then for any essentially smooth $k$-scheme $S$, the projection $\A^1_S\xrightarrow{\pi} S$ induces an isomorphism $$H^1_{\textrm{\'et}}(S,T) \cong H^1_{\text{\'et}}(\A^1_S, T).$$
\end{lemma}
\begin{proof}
We claim that there is an isomorphism of \'etale sheaves $\pi_*(T_{|\A^1_S}) \cong T_{|S}$ on $S$.  By the Leray spectral sequence 
\[
E_2^{p,q} = H^p(S, R^q\pi_*T_{|\A^1_S}) \Rightarrow H^{p+q}(\A^1_S, T_{\A^1_S}),
\]
the required assertion follows from the assertion that $R^1\pi_*T_{|\A^1_S}=0$.  Since this can be checked \'etale locally on $S$, we may assume that $S$ is a strictly henselian scheme.  In this case, the structure map $S\to \Spec(k)$ factors through the separable closure of $\Spec(k_{\rm sep}) \to \Spec (k)$.  Thus, we may assume that $k$ is separably closed.  Since $T_{k_{\rm sep}}$ is split and since $H^1_{\text{\'et}}(\A^1_S, \#G_m) = 0$, it follows that $R^1\pi_*T_{|\A^1_S}=0$.  
\end{proof}

\begin{lemma} 
\label{lemma nistoet}
Let $G$ be any algebraic group over $k$. If the presheaf $$U\mapsto H^1_{\text{\'et}}(U, G)$$  is $\A^1$-invariant, then so is the presheaf
 $$ U\mapsto H^1_{\rm Nis}(U, G).$$
\end{lemma}
 \begin{proof}
It is enough to show that for every essentially smooth $U/k$ the map 
$$ H^1_{\rm Nis}(U, G) \to H^1_{\text{\'et}}(U, G)$$ is injective. Let $\alpha \in H^1_{\rm Nis}(U, G)$ be a class which is in the kernel. Let $P\to U$ be a $G$-torsor representing $\alpha$.  The image of $\alpha$ in $H^1_{\text{\'et}}(U, G)$ is trivial if and only if $P \to U$ has a section.  But then $\alpha$ had to be trivial to begin with.
\end{proof}

\begin{proposition}
\label{prop solvable}
For any solvable algebraic group $B$ over a field $k$, the sheaf $\pi_0^{\A^1}(B)$ is strongly $\A^1$-invariant. Moreover, if $k$ is perfect, then $\pi_0^{\A^1}(B)$ is strictly $\A^1$-invariant.
\end{proposition}
\begin{proof}
We first consider the special case of a torus $T$ over $k$.  Since $T$ is $\A^1$-rigid, $\pi_0^{\A^1}(T) = T$, which is strongly $\A^1$-invariant by Lemma \ref{lemma torus} and Lemma \ref{lemma nistoet}.  

Since any solvable group $B$ is $\A^1$-weakly equivalent to its maximal torus, the general case follows from the special case.  The final assertion follows from \cite[Theorem 5.46]{Morel-book}.
\end{proof}

We next turn our attention to quasi-split groups over perfect fields. Recall that a \emph{Borel subgroup} of an algebraic group $G$ over $k$ is a maximal Zariski closed and connected solvable algebraic subgroup.  A reductive group $G$ is said to be \emph{quasi-split} over $k$ if $G$ admits a Borel subgroup defined over $k$.  In the remainder of this section, we show that the sheaf of $\A^1$-connected components of a quasi-split group is strictly $\A^1$-invariant.  This is shown by obtaining the $\A^1$-connectedness of the generalized flag variety $G/B$ associated with a quasi-split group $G$.   

\begin{lemma}
\label{lemma unipotent}
For any connected unipotent group $G$ over a perfect field $k$ and a subgroup $H$ of $G$, the quotient $G/H$ is $\A^1$-contractible.
\end{lemma}
\begin{proof}
We first observe that the statement clearly holds for $\#G_a$. Over a perfect field, any commutative unipotent group is obtained by successively taking extensions by $\#G_a$ \cite[Theorem 15.4, Corollary 15.5]{Borel}. By induction, the statement of the lemma holds for any commutative unipotent group.  

For the general case, we use induction on the dimension of the group. Let $G$ be any connected unipotent group and let $Z(G)^{\circ}$ denote the identity component of its center.  Since $Z(G)^{\circ}$ is commutative and unipotent, it follows that $Z(G)^{\circ}/Z(G)^{\circ} \cap H$ is $\A^1$-contractible by the special case.  Hence, we have an $\A^1$-fiber sequence
\[
Z(G)^{\circ}/Z(G)^{\circ} \cap H - G/H \to \-G/\-H, 
\]
where $\-G = G/Z(G)^{\circ}$ and $\-H=H/Z(G)^{\circ} \cap H$.  By induction hypothesis, $\-G/\-H$ is $\A^1$-contractible and consequently, so is $G/H$.
\end{proof}

\begin{theorem}
\label{theorem quasi-split}
Let $k$ be a perfect field and let $G$ be a quasi-split algebraic group over $k$.  Fix a Borel subgroup $B$ of $G$ and a maximal torus $T$ of $G$ contained in $B$.  Then $G/B$ is $\#A^1$-connected.  
\end{theorem}
\begin{proof}
Let $W = N_G(T)/T$ denote the absolute Weyl group of $G$, which is a finite \'etale group scheme over $k$.
Let $U=Bw_0B/B$ denote the big cell in the quotient $G/B$ of $G$ by the left action of $B$, where $w_0$ denotes the unique longest element of the Weyl group $W$ of $G$. Since $G$ is quasi-split and $k$ is perfect, the element $w_0$ is defined over $k$ and $U$ is a dense open subscheme of $G/B$. 

The Borel subgroup $B$ is the semidirect product of the centralizer $Z_G(T)$ of $T$ in $G$ and the unipotent radical $R_uB$ of $B$.  The unipotent group $R_uB$ acts transitively on the big cell $U$.  Thus, as a variety over $k$, $U$ is isomorphic to a quotient of the unipotent group $R_uB$ by a subgroup.  By Lemma \ref{lemma unipotent}, $U$ is $\A^1$-connected.  
Then the generic point of $U$ maps to a $k$-rational point of $U$ in $\pi_0^{\A^1}(U)(k(U))$.  But this implies that the generic point of $G/B$ maps to a $k$-rational point of $G/B$ in $\pi_0^{\A^1}(G/B)(k(G/B))$, $U$ being an open subscheme of $G/B$.  Thus, we conclude by \cite[Theorem 2]{Balwe-Hogadi-Sawant-JAG} that $G/B$ is $\A^1$-connected. This completes the proof of the theorem.
\end{proof}

\begin{corollary}
With the notations of Theorem \ref{theorem quasi-split}, $\pi_0^{\A^1}(G)$ is a strictly $\A^1$-invariant sheaf.
\end{corollary}
\begin{proof}
Note that $\pi_0^{\A^1}(B) = T$ is strictly $\A^1$-invariant by Proposition \ref{prop solvable}.  Hence, by \cite[Thm 6.50]{Morel-book}, the sequence 
\[
B \to G \to G/B 
\]
is an $\A^1$-homotopy principal $B$-fibration. This implies the following long exact sequence of pointed Nisnevich sheaves of groups
\[
\cdots \to \pi_1^{\A^1}(G/B) \to \pi_0^{\A^1}(B)=T \to \pi_0^{\A^1}(G) \to \pi_0^{\A^1}(G/B).
\]
By Theorem \ref{theorem quasi-split}, it follows that $\pi_0^{\A^1}(G)$ is a quotient of $T$ by the image of $\pi_1^{\A^1}(G/B) \to \pi_0^{\A^1}(B)=T$ and that we have an isomorphism $\pi_1^{\A^1}(G/B) \cong \HA_1(G/B)$. Since $\HA_1(G/B)$ is a strictly $\A^1$-invariant sheaf and the category of strictly $\A^1$-invariant sheaves is abelian, it follows that $\pi_0^{\A^1}(G)$ is a strictly $\A^1$-invariant sheaf.

\end{proof}

%

\section{Transfers on cellular \texorpdfstring{$\A^1$}{A1}-homology}
\label{section cellular transfers}

In this section, we show that for a large class of smooth schemes with a cellular structure, the cellular $\#A^1$-chain complex studied in \cite{Morel-Sawant-pi1} has geometric transfers in the sense of \cite{Morel-book}. As a consequence, we will show that the sheaf of $\#A^1$-fundamental groups of the flag variety of a quasi-split group admits geometric transfers, which will be a key tool in the proof of our main theorem.

\begin{definition}
\label{defn naively cellular}
Let $k$ be a field and let $X \in Sm_k$ be a smooth $k$-scheme. A \emph{naive cellular structure} on $X$ consists of an increasing filtration 
\begin{equation}
\label{eqn cellular}
\emptyset = \Omega_{-1} \subsetneq \Omega_0 \subsetneq \Omega_1 \subsetneq \cdots \subsetneq \Omega_s = X 
\end{equation}
by open subschemes such that for each $i\in\{0,\dots,s\}$, the reduced induced closed subscheme $X_i:= \Omega_i - \Omega_{i-1}$ of $\Omega_i$ is either 
\begin{itemize}
\item empty, or
\item everywhere of codimension $i$ and a disjoint union of products of $\#A^1_{L_{ij}}$'s and $({\#G_m})_{L_{ij}}$'s, for some finite separable field extensions $L_{ij}$'s of $k$.
\end{itemize}
We call $X$ \emph{naively cellular} if $X$ is endowed with a naive cellular structure.
\end{definition}

\begin{example}
Let $G$ be a split group over a field $k$ with a fixed Borel subgroup $B$ and a maximal torus $T$ contained in $B$. Then $G$, as well as the flag variety $G/B$ and the quotient $G/T$ are examples of naively cellular schemes in the sense of Definition \ref{defn naively cellular}, where $L_{ij} = k$, for all $i, j$.  The naive cellular structure is induced by the Bruhat decomposition. 
\end{example}

\begin{example}
Let $G$ be a quasi-split group over a field $k$ with a fixed Borel subgroup $B$ (defined over $k$) and a maximal torus $T$ contained in $B$. The flag variety $G/B$ is an example of a naively cellular scheme.  The naive cellular structure in this case is induced by the absolute Bruhat decomposition by considering the scheme-theoretic Bruhat decomposition of the flag variety $(G/B)_{\bar{k}}$ of the split group $G_{\bar{k}}$ and by taking the fixed points under the action of $\Gal(\bar{k}/k)$. Note that in this case the Weyl group $W = N_G(T)/T$ is a finite \'etale group scheme.  Hence, the Bruhat cells in $(G/B)_{\bar{k}}$, which are isomorphic to affine spaces over $\bar{k}$ need not descend to give affine spaces over $k$ as locally closed strata. However, the strata are isomorphic to affine spaces over a finite separable field extension of $k$. We will use this naive cellular structure in the proof of Theorem \ref{theorem transfers}.
\end{example}

Let $X$ be a naively cellular smooth scheme over $k$ admitting a filtration as in \eqref{eqn cellular}. Write $X_i = \Omega_i - \Omega_{i-1} = \coprod_j Y_{ij}$, where each $Y_{ij}$ is a product of $\#A^1$'s and $({\#G_m})$'s, possibly defined over finite separable field extensions of $k$. Then following \cite[Section 2.3]{Morel-Sawant-pi1}, the cellular $\#A^1$-chain complex of $X$ is given by
\[
\Ccell_*(X): \cdots \to \HAred_i(\Omega_i/\Omega_{i-1}) \xrightarrow{\partial_i} \to  \HAred_{i-1}(\Omega_{i-1}/\Omega_{i-2}) \cdots, 
\]
where $\HAred$ denotes the reduced $\#A^1$-homology groups and the differential $\partial_i$ is defined to be the composite
\[
\HAred_i(\Omega_i/\Omega_{i-1}) \xrightarrow{\delta} \HAred_{i-1}(\Omega_{i-1}) \xrightarrow{\iota} \HAred_{i-1}(\Omega_{i-1}/\Omega_{i-2}),
\]
where $\delta$ is the connecting map in the long exact sequence of $\#A^1$-homology groups associated with the cofiber sequence $\Omega_{i-1} \to \Omega_i \to \Omega_i/\Omega_{i-1}$ and $\iota$ is induced by the quotient morphism $\Omega_{i-1} \to \Omega_{i-1}/\Omega_{i-2}$.  Note that by hypothesis, each $Y_{ij}$ is cohomologically trivial in the sense of \cite[Definition 2.9, Remark 2.10]{Morel-Sawant-pi1}. By \cite[Lemma 2.13]{Morel-Sawant-pi1}, the normal bundle of the inclusion of each $Y_{ij}$ is trivial. Hence, as an application of the motivic homotopy purity theorem \cite[Theorem 2.23, page 115]{Morel-Voevodsky}, we have the identification
\[
\HAred_i(\Omega_i/\Omega_{i-1})  \cong \underset{j}{\oplus} ~ \HAred_i(S^i \wedge \#G_m^{\wedge i} \wedge (Y_{ij})_+) \cong \underset{j}{\oplus} ~ \KMW_i \otimes \ZA[Y_{ij}].
\] 
By hypothesis, each $Y_{ij}$ is $\#A^1$-weakly equivalent to $T_{ij}:= ({\#G_m})_{L_{ij}}^{\times r_{ij}}$ for some nonnegative integer $r_{ij}$ and some finite separable extension $L_{ij}$ of $k$.  Let us denote the symmetric monoidal structure on $Ab_{\#A^1}(k)$ by $\otimes$.  Recall from \cite[Section 3.3]{Morel-Sawant-pi1} that there is an isomorphism of sheaves 
\begin{equation}
\label{eqn ZT decomposition}
\ZA[T_{ij}]  \cong \ZA[\#G_m]^{\otimes r_{ij}}\otimes \Z[\Spec L_{ij}] = \Z[\Spec L_{ij}] \oplus \left( \underset{1\leq i_1<\dots<i_s\leq r_{ij}}{\oplus} \KMW_s \otimes \Z[\Spec L_{ij}]\right).
\end{equation}
By \cite{Morel-Sawant-WR}, one has an isomorphism
\begin{equation}
\label{eqn Weil restriction}
\KMW_n \otimes \Z[\Spec L] \cong {f_L}_*f_L^* \KMW_n, 
\end{equation}
for every $n \in \Z$, and any finite field extension $L$ of $k$, where $f_L: \Spec L \to \Spec k$ is the structure morphism.

\begin{notation}
For the sake of brevity, we write ${\KMW_n}^{[L]}: = {f_L}_*f_L^* \KMW_n$.
\end{notation}

Given a monogenic extension $F \subset E$ of finitely generated, separable field extensions of $k$ along with a generator $\alpha \in E$ of $E$ over $F$, we have a geometric transfer map \cite[page 98]{Morel-book}
\[
tr^{\alpha}_n: \Ccell_n(X)(E) \to \Ccell_n(X)(F),
\]
for every integer $n$. Composing suitable geometric transfer maps, we get field transfers on the terms of $\Ccell_*(X)$.

\begin{definition}
\label{defn field transfers}
Let $M \in Ab_{\#A^1}(k)$. Let $F \subset E$ be an extension of finitely generated, separable field extensions of $k$. Choose a sequence of field extensions
\[
F = E_0 \subset E_1 \subset \cdots E_r = E, 
\]
where each $E_i/E_{i-1}$ is a monogenic extension and choose a generator $\alpha_i$ of $E_i$ over $E_{i-1}$. By a \emph{field transfer} on $M_{-1}$, we mean a composite of geometric transfers 
\[
M_{-1}(E) \xrightarrow{tr^{\alpha_r}} M_{-1}(E_{r-1}) \to \cdots \to M_{-1}(E_1) \xrightarrow{tr^{\alpha_1}} M_{-1}(F). 
\]
Note that a field transfer depends upon the choices of the tower of monogenic extensions and the generators.
\end{definition}

Now, assume that $X$ is naively cellular with the additional condition that $L_{ij} = k$, for all $i, j$ in Definition \ref{defn naively cellular}. In this case, it follows that every term of the cellular $\#A^1$-chain complex $\Ccell_*(X)$ of $X$ is a direct sum of the sheaves $0$, $\#Z$ and the unramified Milnor-Witt $K$-theory sheaves $\KMW_n$'s. Recall that for every $M \in Ab_{\#A^1}(k)$ and $n \in \#Z$, we have the identification
\begin{equation}
\label{eqn Hom from KMW}
\Hom_{Ab_{\#A^1}(k)}(\KMW_n, M) = M_{-n}(k).
\end{equation}
Using this, it follows that the components of the differentials of $\Ccell_*(X)$ are either $0$, or elements of $\#Z$ or $\KMW_n(k)$ for some $n \in \#Z$. More precisely, the components of the form $\KMW_n \to \KMW_{n-1}$ for $n > 1$ are of the form $\eta$ multiplied by an element of $\KMW_0(k) = GW(k)$.  

The above discussion along with the projection formula \cite[Remark 4.30]{Morel-book} shows that field transfers commute with the differentials of $\Ccell_*(X)$ in this special case and hence, induce field transfers on the cellular $\#A^1$-homology groups. We record this observation below.

\begin{theorem}
\label{theorem cellular transfers}
Let $X$ be a naively cellular smooth scheme over a perfect field $k$ with the additional condition that $L_{ij} = k$, for all $i, j$ in Definition \ref{defn naively cellular}.  Then its cellular $\A^1$-chain complex $\Ccell_*(X)$ admits field transfers, which respect differentials.  Consequently, the cellular $\A^1$-homology sheaves admit field transfers. 
\end{theorem}

\begin{corollary}
\label{corollary G/T}
Let $G$ be a split group over a perfect field $k$. Fix a Borel subgroup $B$ of $G$. Then the sheaves $\Hcell_n(G)$, $\Hcell_n(G/B)$ admit field transfers for all $n \in \#N$.
\end{corollary}
\begin{proof}
Under the hypotheses on $G$, we have shown in Theorem \ref{theorem quasi-split} that $G/B$ is $\#A^1$-connected. Hence, by \cite{Morel-Sawant-pi1}, we have isomorphisms
\[
\pi_1^{\A^1}(G/B) \cong \HA_1(G/B) \cong \Hcell_1(G/B).
\]
Therefore, the corollary follows from Theorem \ref{theorem cellular transfers}.
\end{proof}

\section{Transfers on the \texorpdfstring{$\A^1$}{A1}-connected components of quasi-split groups}

Recall that a \emph{sheaf with (Voevodsky) transfers} over a field $k$ is a contravariant functor from the category of finite correspondences over $k$ to the category of abelian groups (see \cite[Chapter 2]{MVW}). The aim of this section is to give a proof of the main theorem (Theorem \ref{theorem intro main} = Theorem \ref{theorem transfers}), which says that the sheaf of $\A^1$-connected components of a quasi-split group over a perfect field is a presheaf with transfers. 

\begin{lemma}
\label{lemma torus subsheaf} 
Let $T$ be a torus over a field $k$. Then any strictly $\A^1$-invariant Nisnevich subsheaf $\@F$ of $T$, which is stable under field transfers, is stable under transfers.
\end{lemma}
\begin{proof}
Recall that $T$ is a strictly $\A^1$-invariant sheaf with transfers. Let $Z$ be a finite prime correspondence from $X$ to $Y$, that is, $Z$ is an irreducible closed subscheme of $X \times Y$ that is finite over $X$. We have the transfer map $Z^*: T(Y) \to T(X)$ and an injection $T(X) \hookrightarrow T(k(X))$. 

Let $k(Z)$ denote the function field of $Z$. Since $k(Z)$ is a finite extension of the function field $k(X)$ of $X$, there exists a field transfer map $T(k(Z)) \xrightarrow{tr} T(k(X))$. Since $\@F$ is stable under field transfers, we get the following commutative diagram 
\[
\begin{xymatrix}{
& \@F(X) \ar@{^{(}->}[dr] & \\
\@F(Y) \ar[r] \ar@{^{(}->}[d] & \@F(k(Z))\ar@{^{(}->}[d] \ar[r]^{tr} & \@F(k(X)) \ar@{^{(}->}[d]\\
T(Y) \ar[dr]^-{Z^*} \ar[r] & T(k(Z)) \ar[r]^{tr} & T(k(X)) \\
& T(X)\ar@{^{(}->}[ur] &}
\end{xymatrix}
\]
with injective vertical maps. We need to show that there exists a transfer map $\@F(Y) \to \@F(X)$ making the outer diagram commute. This will follow if we show that 
\[
T(X) \cap \@F(k(X)) = \@F(X),
\]
where all the groups are seen as subgroups of $T(k(X))$. Since $\@F$ and $T$ are both strictly $\#A^1$-invariant, we have the commutative diagram
\[
\begin{xymatrix}{
0 \ar[r] & \@F(X) \ar[r] \ar@{^{(}->}[d] & \@F(k(X))\ar@{^{(}->}[d] \ar[r]  & \underset{x \in X^{(1)}}{\oplus} \@F_{-1}(k(x)) \ar@{^{(}->}[d]\\
0 \ar[r] & T(X) \ar[r] & T(k(X)) \ar[r] & \underset{x \in X^{(1)}}{\oplus} T_{-1}(k(x))
}
\end{xymatrix}
\]
where the vertical maps are all injective (since contraction is an exact functor), the right horizontal maps are the differentials of the Rost-Schmid complexes of $\@F$ and $T$, and the rows are exact (by \cite[Corollary 5.43]{Morel-book}). Clearly, $T(X) \cap \@F(k(X))$ contains $\@F(X)$. Conversely, if an element of $T(k(X))$ is in the image of both $T(X)$ and $\@F(k(X)$, then its image in $\underset{x \in X^{(1)}}{\oplus} T_{-1}(k(x))$ is $0$. Since the rightmost vertical map is injective, it follows that its image in $\underset{x \in X^{(1)}}{\oplus} \@F_{-1}(k(x))$ is $0$. By the exactness of the top row, it follows that this element comes from $\@F(X)$. This completes the proof.
\end{proof}

\begin{lemma}
\label{lemma transfers}
Let $T$ be a torus over a field $k$. Then any homomorphism $\KMW_1 \to T$ of strictly $\A^1$-invariant sheaves preserves transfers.
\end{lemma}
\begin{proof}
By \cite[Theorem 3.37]{Morel-book}, it follows that there is a bijection
\begin{equation}
\label{eqn KMW}
\Hom_{Ab_{\#A^1}(k)}(\KMW_1, T) = \Hom_{Shv_{\bullet}(Sm/k)}(\#G_m, T).
\end{equation}
However, any pointed map of Nisnevich sheaves of sets $\#G_m \to T$ is just a pointed morphism of schemes $\#G_m \to T$. We claim that it is a group homomorphism and preserves transfers. To see this, by \'etale descent, one reduces to the case $T = \#G_m$ and then uses the fact that the units of the ring $k[x, x^{-1}]$ are always of the form $cx^n$, for some $c \in k^{\times}$ and $n \in \#Z$. 

We next claim that any morphism $\KMW_1 \xrightarrow{\alpha} T$ of strictly $\A^1$-invariant sheaves factors through the canonical epimorphism $\KMW_1 \to \#G_m$. Let $\~\alpha: \#G_m \to T$ denote the morphism corresponding to $\alpha$ under the bijection \eqref{eqn KMW}. It suffices to prove the claim for sections over field extensions of $k$.  For any finitely generated, separable field extension $L/k$, we have a commutative diagram  
\[
\begin{xymatrix}{
\KMW_1(L) \ar[r]^-{\alpha} & T(L) \\
\#G_m(L) \ar[u] \ar[ur]_-{\~\alpha} &}
\end{xymatrix}
\]
where the vertical morphism is the canonical map sending a unit $a \in L^\times$ to the symbol $[a]$. Since $\~\alpha$ is a group homomorphism, it follows that 
\[
\alpha (\eta [a][b]) = 0,
\]
for every $a, b \in L^\times$. Therefore, the claim follows.
 
Now, since both the morphisms $\KMW_1 \to \#G_m$ and $\#G_m \to T$ preserve transfers, the lemma follows.
\end{proof}

For the proof of our main theorem, we will need the following extension of Lemma \ref{lemma transfers}.

\begin{lemma}
\label{lemma WR transfers}
Let $T$ be a torus over a field $k$ and let $L$ be a finite separable extension of $k$. Then any homomorphism ${\KMW_1}^{[L]} \to T$ of strictly $\A^1$-invariant sheaves preserves transfers.
\end{lemma}
\begin{proof}
The canonical epimorphism $({\KMW_1})_L \to {\#G_m}_L$ of strictly $\A^1$-invariant sheaves over $L$ induces a morphism ${\KMW_1}^{[L]} \to {\#G_m}^{[L]}$, which admits a canonical section $\tau$, induced by the canonical morphism of sheaves of sets ${\#G_m}_L \to ({\KMW_1})_L$. Note that $\tau$ is just a morphism of the underlying sheaves of sets.

Given a morphism $\alpha: {\KMW_1}^{[L]} \to T$, define $\~\alpha: {\#G_m}^{[L]} \to T$ by $\~\alpha := \alpha \circ \tau$. We claim that the diagram 
\[
\begin{xymatrix}{
{\KMW_1}^{[L]} \ar[d] \ar[r]^-{\alpha} & T \\
\#G_m^{[L]}  \ar[ur]_-{\~\alpha} &}
\end{xymatrix}
\]
commutes. It suffices to verify this for sections over field extensions of $k$, since ${\KMW_1}^{[L]}$, $\#G_m^{[L]}$ and $T$ are all strictly $\A^1$-invariant sheaves over $k$.  Let $E$ be a finitely generated, separable field extension of $k$. Let $a, b \in \#G_m^{[L]}(E) = (L \otimes_k E)^{\times}$. Since $k$ is perfect, $L \otimes_k E$ is a product of finite field extensions of $E$. One easily reduces to the case where $L \otimes_k E$ is a field. By a similar argument as in the proof of Lemma \ref{lemma transfers}, it follows that $\~\alpha$ is a group homomorphism. Therefore, we have
\[
\alpha([ab]) = \alpha([a]) + \alpha([b]).
\]
But then $\alpha(\eta [a][b]) = 0$, since $\alpha$ is a homomorphism of strictly $\A^1$-invariant sheaves. This proves the claim. 

We end the proof along similar lines as in the proof of Lemma \ref{lemma transfers}.  The canonical morphism ${\KMW_1}^{[L]} \to {\#G_m}^{[L]}$ and $\~\alpha$ preserve transfers. Hence, so does $\alpha$.
\end{proof}

We are now set to prove the main theorem of this paper.

\begin{theorem}
\label{theorem transfers}
Let $G$ be a quasi-split group over a perfect field $k$. Fix a Borel subgroup $B$ of $G$ and the corresponding maximal torus $T$ of $G$.  Then the composition of natural maps $T \to G \to \pi_0^{\A^1}(G)$ is an epimorphism of Nisnevich sheaves.  Moreover, the sheaf $\pi_0^{\A^1}(G)$ is a strictly $\A^1$-invariant sheaf with transfers. 
\end{theorem}
\begin{proof}
We have an $\#A^1$-fibration $T \to G \to G/T$ and the corresponding long exact sequence of $\A^1$-homotopy sheaves has the form
\[
\cdots \to \pi_1^{\#A^1}(G/T) \xrightarrow{\theta} T \to \pi_0^{\#A^1}(G) \to \pi_0^{\#A^1}(G/T). 
\]
By Theorem \ref{theorem quasi-split}, $G/T$ is $\#A^1$-connected, so we have an exact sequence
\[
\pi_1^{\#A^1}(G/T) \xrightarrow{\theta} T \to \pi_0^{\#A^1}(G) \to 0. 
\]
Since $T$ has transfers, in order to show that $\pi_0^{\#A^1}(G)$ has transfers, it suffices to show that the image of $\theta$ is stable under transfers of $T$. By Lemma \ref{lemma torus subsheaf}, it suffices to show that the image of $\theta$ is stable under field transfers.

Consider the absolute Bruhat decomposition of $G/B$, which is $\A^1$-weakly equivalent to $G/T$. As described in Section \ref{section cellular transfers}, the cellular $\A^1$-chain complex of $G/T$ has the form
\[
\cdots \to \Ccell_2(G/T) \to \Ccell_1(G/T) = \oplus_i ~ {\KMW_1}^{[L_i]} \xrightarrow{\partial_1}  \Ccell_0(G/T) = \#Z.
\]
We claim that the differential $\partial_1$ is the zero map. Indeed, this can be verified on every finitely generated, separable field extension $F$ over $k$. Now, for each $i$, we have ${\KMW_1}^{[L_i]}(F) = {\KMW_1}(L_i \otimes_k F)$. Since $k$ is perfect, $L_i \otimes_k F$ is a direct product of finite separable field extensions of $F$ and the claim follows since any morphism from $\KMW_1$ to a constant sheaf has to be the zero map.

Therefore, we have an epimorphism $\Ccell_1(G/T) \twoheadrightarrow \Hcell_1(G/T)$.  Theorem \ref{theorem quasi-split} implies that $\Hcell_1(G/T) \cong \pi_1^{\#A^1}(G/T)$. Hence, the image of $\theta$ is the same as the image of the composite
\[
\Ccell_1(G/T) \twoheadrightarrow \pi_1^{\#A^1}(G/T) \xrightarrow{\theta} T.
\]
By Lemma \ref{lemma WR transfers}, this composite preserves field transfers. This completes the proof of the theorem.
\end{proof}

\section{Applications to the norm principle}

Let $G$ be a reductive algebraic group over a field $k$. The goal of this section is to show that existence of transfer maps on $\pi_0^{\A^1}(G)$ implies that the norm principle holds for $G$. We first recall the following definition of norm principle.

\begin{definition}(\cite[page 2]{barquero-merkurjev})
\label{defn norm principle}
Let $H$ be a torus over $k$. Let $\phi:G\to H$ be a homomorphism. For any separable extension $L/F$ of $k$ fields we have the following diagram 
$$\xymatrix{
G(L) \ar[r]^{\phi_L} & H(L) \ar[d]^{N_{L/F}} \\
G(F) \ar[r]^{\phi_F} & H(F)
}$$
where the right vertical map $N_{L/F}$ is the norm (or the field transfer) map. We say that the norm principle holds for $\phi$, if for any $L/F$ as above, we have 
$$ {\rm Image}(N_{L/F}\circ \phi_L) \subset \phi_K.$$
We say the norm principle holds for $G$ if it holds for any homomorphism from $G$ to a torus over $k$. 
\end{definition}

\begin{remark}
The version of the norm principle given in Definition \ref{defn norm principle} is due to Merkurjev \cite{Merkurjev-norm-principle}. Although this version appears to be different from the version given by Gille \cite{Gille-1993}, which is stated in terms of torsors for the group in question, the two versions of the norm principle are equivalent (see \cite[Theorem 2.1]{Soofiani}).
\end{remark}

\begin{theorem}
\label{theorem pi0norm}
Let $G$ be any a connected reductively algebraic group over a field $k$ such that $\pi_0^{\A^1}(G)$ has transfers. Then the norm principle holds for $G$. 
\end{theorem}
\begin{proof} Let $\phi : G\to H$ be any homomorphism from $G$ to a torus $H$ and $L/F$ be a separable extension of fields over $k$. Since $H$  is a torus the natural map 
$$\pi_0^{\A^1}(H) \xrightarrow{\sim} H$$
is an isomorphism.  We thus have a commutative diagram 
$$\xymatrix{
G \ar[d] \ar[r]^{\phi} & H \ar[d]^{\sim}\\
\pi_0^{\A^1}(G) \ar[r] & \pi_0^{\A^1}(H)
}$$
We first claim that $\pi_0^{\A^1}(G)\to H$ preserves transfers. To see this, let $T\subset G$ be a maximal torus. By Theorem \ref{theorem transfers}, the natural map $T\to \pi_0^{\A^1}(G)$ is an epimorphism of sheaves which preserves transfers. Thus it is enough to show that $T\to H$ preserves transfers. But this is clear since $T\to H$ is a homomorphism of tori. Also note that $G\to \pi_0^{\A^1}(G)$ is an epimorphism of Nisnevich sheaves. In particular, it is epimorphism on sections over any field. We have the following commutative diagram: 
$$\xymatrix{
G(L) \ar@{>>}[r] & \pi_0^{\A^1}(G)(L)\ar[d]^{tr_{L/F}} \ar[r] & H(L) \ar[d]^{N_{L/F}} \\
G(F) \ar@{>>}[r] & \pi_0^{\A^1}(G)(F) \ar[r] & H(F).
}$$
The norm principle for $\phi$ now directly follows.
\end{proof}

\begin{corollary}
\label{corollary norm principle}
The norm principle holds for any quasi-split group over a perfect field $k$.
\end{corollary}
\begin{proof}
This is a direct consequence of Theorem \ref{theorem transfers} and Theorem \ref{theorem pi0norm}. 
\end{proof}


\begin{thebibliography}{9999}

\bibitem{Balwe-Hogadi-Sawant-JAG}
C. Balwe, A. Hogadi, A. Sawant:
\emph{Geometric criteria for $\A^1$-connectedness and applications to norm varieties}, J. Algebraic Geom., 32 (2023), no. 4, 677--696.

\bibitem{Balwe-Hogadi-Sawant-JTop}
C. Balwe, A. Hogadi, A. Sawant:
\emph{Strong $\A^1$-invariance of $\A^1$-connected components of reductive algebraic groups},
Journal of Topology, 16 (2023), no. 2, 634--649.

\bibitem{Balwe-Sawant-R-eqivalence-IMRN}
C. Balwe and A. Sawant:
\emph{$R$-equivalence and $\A^1$-connectedness in anisotropic groups},
Int. Math. Res. Not. IMRN 2015, No. 22, 11816--11827.

\bibitem{Balwe-Sawant-reductive}
C. Balwe, A. Sawant:
\emph{$\mathbb A^1$-connectedness in reductive algebraic groups},
Trans. Amer. Math. Soc. 369 (2017), no. 8, 5999--6015.

\bibitem{barquero-merkurjev}
P. Barquero, A.S. Merkurjev:
\emph{Norm principle for reductive algebraic groups}, Algebra, arithmetic and geometry, Part I, II (Mumbai, 2000), 123--137.
Tata Inst. Fund. Res. Stud. Math., 16.

\bibitem{Bhaskhar-Chernousov-Merkurjev}
N. Bhaskhar, V. Chernousov, A.S. Merkurjev:
\emph{The norm principle for type $D_n$ groups over complete discretely valued fields},
Trans. Amer. Math. Soc. 372 (2019), no. 1, 97--117.

\bibitem{Borel}
A. Borel:
\emph{Linear algebraic groups},
Second edition, Grad. Texts in Math., 126, Springer-Verlag, New York, 1991. 

\bibitem{Choudhury}
U. Choudhury:
\emph{Connectivity of motivic H-spaces},
Algebr. Geom. Topol. 14 (2014), no. 1, 37--55.

\bibitem{Gille-1993}
P. Gille:
\emph{$R$-\'equivalence et principe de norme en cohomologie galoisienne},
C. R. Acad. Sci. Paris S\'er. I Math. 316 (1993), no. 4, 315--320.

\bibitem{Gille-1997}
P. Gille:
\emph{La  $R$-\'equivalence sur les groupes alg\'ebriques r\'eductifs d\'efinis sur un corps global},
Inst. Hautes \'Etudes Sci. Publ. Math. No. 86 (1997), 199--235.

\bibitem{MVW}
C. Mazza, V. Voevodsky, C. Weibel:
\emph{Lecture notes on motivic cohomology}, Clay Math. Monogr., 2
American Mathematical Society, Providence, RI; Clay Mathematics Institute, Cambridge, MA, 2006.

\bibitem{Merkurjev-norm-principle}
A.S. Merkurjev:
\emph{The norm principle for algebraic groups},
Algebra i Analiz 7 (1995), no. 2, 77--105; translation in
St. Petersburg Math. J. 7 (1996), no. 2, 243--264.

\bibitem{Morel-book}
F. Morel: 
\emph{$\mathbb A^1$-algebraic topology over a field},
Lecture Notes in Mathematics, Vol. 2052, Springer, Heidelberg, 2012.
 
\bibitem{Morel-connectivity}
F. Morel:
\emph{The stable $\mathbb A^1$-connectivity theorems},
K-Theory 35 (2005), 1--68.

\bibitem{Morel-Sawant-pi1}
F. Morel, A. Sawant:
\emph{Cellular $\mathbb A^1$-homology and the motivic version of Matsumoto's theorem},
Adv. Math. 434 (2023), Paper No. 109346, 110 pp. 

\bibitem{Morel-Sawant-WR}
F. Morel, A. Sawant:
\emph{Weil restrictions, cellular structures and Milnor-Witt $K$-theory}, in preparation.

\bibitem{Morel-Voevodsky}
F. Morel, V. Voevodsky:
\emph{$\A^1$-homotopy theory of schemes},
Inst. Hautes \'Etudes Sci. Publ. Math. 90 (1999), 45--143.

\bibitem{Soofiani}
A. Soofiani:
\emph{Hensel's lemma for the norm principle for spinor groups},
Preprint, arXiv:2412.15737v1 [math.GR].

\end{thebibliography}
\end{document}